\documentclass[11pt]{amsart}

\usepackage{mathpazo}
\usepackage[top=25mm, bottom=25mm, left=25mm, right=25mm]{geometry} 

\usepackage{amsmath,amssymb}
\usepackage{graphicx}
\usepackage{amscd}
\usepackage{wrapfig}
\usepackage{float}
\usepackage{epsfig}
\usepackage{tikz}

\newtheorem{theorem}{Theorem}[section]

\newtheorem{corollary}[theorem]{Corollary}
\newtheorem{proposition}[theorem]{Proposition}

\newtheorem{definition}[theorem]{Definition}

\DeclareMathOperator{\teich}{Teich}
\DeclareMathOperator{\thick}{Thick}
\DeclareMathOperator{\Mod}{Mod}
\DeclareMathOperator{\len}{Length}
\DeclareMathOperator{\sys}{sys}
\DeclareMathOperator{\shad}{Shad}
\DeclareMathOperator{\pmf}{PMF}
\DeclareMathOperator{\Nbhd}{Nbhd}
\DeclareMathOperator{\isom}{Isom}

\newcommand{\N}{{\mathbb N}}
\newcommand{\R}{{\mathbb R}}
\newcommand{\Z}{{\mathbb Z}}
\newcommand{\T}{{\mathcal T}}
\renewcommand{\H}{{\mathbb H}}
\newcommand{\cC}{\mathcal{C}}

\newcommand{\teichmuller}{Teichm{\"u}ller{ }}

\begin{document}
	
	\title{Statistical hyperbolicity for harmonic measure.}

\author{Aitor Azemar}
	\address{School of Mathematics and Statistics, University of Glasgow, University Place, Glasgow G12 8QQ}
	\email{Aitor.Azemar@glasgow.ac.uk}
	
	\author{Vaibhav Gadre}
	\address{School of Mathematics and Statistics, University of Glasgow, University Place, Glasgow G12 8QQ}
	\email{Vaibhav.Gadre@glasgow.ac.uk}
	
	\author{Luke Jeffreys}
	\address{School of Mathematics and Statistics, University of Glasgow, University Place, Glasgow G12 8QQ}
	\email{l.jeffreys.1@research.gla.ac.uk}
	
	\keywords{\teichmuller theory, Moduli of Riemann surfaces.}
	\subjclass[2010]{30F60, 32G15, 60G50}

	\begin{abstract}
	We consider harmonic measures that arise from random walks on the mapping class group determined by probability distributions that have finite first moment with respect to the \teichmuller metric, and whose supports generate non-elementary subgroups.
	We prove that \teichmuller space with the \teichmuller metric is statistically hyperbolic for such a harmonic measure.

	\end{abstract}

	\maketitle

	\section{Introduction}
	The notion of \emph{statistical hyperbolicity}, introduced by Duchin-Leli{\`e}vre-Mooney~\cite{Duc-Lel-Moo}, encapsulates whether a space is on average hyperbolic at large scales, that is, for any point in the space and spheres centred at that point whether as the radius $r \to \infty$ the average distance between pairs of points on the sphere of radius $r$ is $2r$. 
	To make sense of the average distance, one requires reasonable measures on spheres.
	
	For many Lebesgue-class measures on \teichmuller space, Dowdall-Duchin-Masur showed that \teichmuller space with the \teichmuller metric is statistically hyperbolic. 
	See \cite[Theorems B, C and D]{Dow-Duc-Mas}. 
	See also \cite{Hid}. 
	Here, we consider the same question for harmonic measures that arise from random walks on the mapping class group determined by probability distributions with finite first moment with respect to the \teichmuller metric, and whose supports generate non-elementary subgroups.
	
	Kaimanovich-Masur showed that a random walk on the mapping class group, whose initial support generates a non-elementary subgroup, converges to the Thurston boundary of \teichmuller space with probability one. 
	This defines a harmonic measure on the Thurston boundary and Kaimanovich-Masur showed that this measure is supported on the set of uniquely ergodic measured foliations. 
	See \cite[Theorem 2.2.4]{Kai-Mas1} for both statements.
	Since \teichmuller rays with uniquely ergodic vertical foliations asymptotically converge to this vertical foliation, it is possible to pull back the harmonic measure to the unit cotangent space at a base-point. 
	This allows us to equip spheres in \teichmuller space with a harmonic measure.
	We can then consider the question of whether \teichmuller space is statistically hyperbolic with respect to these measures.
	
	Our main theorem is:
	
	\begin{theorem}\label{t:main}
		Let $S$ be a surface of finite type. 
		Let $\mu$ be a probability distribution on the mapping class group $\Mod(S)$ with finite first moment with respect to the \teichmuller metric, and such that the support generates a non-elementary subgroup.
		Then \teichmuller space $\T(S)$ with the \teichmuller metric is statistically hyperbolic with respect to the harmonic measure defined by the $\mu$-random walk on $\Mod(S)$. 
	\end{theorem}
	
	When $S$ is a torus or a torus with one marked point or a sphere with four marked points, $\T(S)$ with the \teichmuller metric is isometric to $\H$. 
	When $\mu$ has finite first moment in the word metric then by a theorem of Guivarch-LeJan \cite{Gui-LeJ}, the harmonic measure from the $\mu$-random walk is singular with respect to the Lebesgue measure class. 
	See also \cite{Bla-Hai-Mat}, \cite{Der-Kle-Nav} and \cite{Gad-Mah-Tio2} for other proofs. 
	
	With respect to the class of Lebesgue measures on $\T(S)$, a similar singularity of harmonic measures also holds when the complex dimension is greater than one. 
	See~\cite[Theorem 1.1]{Gad} for singularity of harmonic measures from finitely supported random walks on $\Mod(S)$ and \cite[Theorem 1.4]{Gad-Mah-Tio} for singularity for harmonic measures for random walks with finite first moment with respect to a word metric on $\Mod(S)$. 
	Finite first moment with respect to the word metric implies finite first moment with respect to the \teichmuller metric. 
	For this large class of random walks, Thereom~\ref{t:main} gives a conclusion that is distinct from the main results of Dowdall-Duchin-Masur~\cite{Dow-Duc-Mas} and, as we outline below, requires different tools.
	
	On the other hand, finite first moment with respect to the \teichmuller metric does not imply finite first moment in the word metric. 
	This is because the mapping class group is distorted under the orbit map to \teichmuller space.
	For the exceptional surfaces mentioned above whose \teichmuller space is $\H$, Furstenberg showed that there is a finite first $d_{\H}$-moment random walk whose harmonic measure on $S^1$ is absolutely continuous. 
	Thus, for the exceptional surfaces Theorem~\ref{t:main} derives statistical hyperbolicity covering both singular and Lebesgue class measures in one statement. 
	For non-exceptional moduli, a solution to the Furstenberg problem has recently been announced by Eskin-Mirzakhani-Rafi \cite{Esk-Mir-Raf}. 
	Thus Theorem~\ref{t:main} covers singular and also Lebesgue class measures analysed by Dowdall-Duchin-Masur in one statement. 
	
	We will first present the proof of Theorem~\ref{t:main} when the complex dimension of $\T(S)$ is greater than one.
	This is the harder case.
	For the exceptional surfaces, that is when $\T(S) = \H$, the proof of Theorem~\ref{t:main} is obviously easier because the ambient geometry is already hyperbolic. 
	However, as mentioned above, many harmonic measures are singular.
	So there is something to prove.
	The argument required in the exceptional case is straightforward and uses the geodesic separation property for harmonic measure that is already formulated in the proof of the harder case of Theorem~\ref{t:main}. 
	
	In fact, we present the exceptional case as a special case of a more general theorem when the ambient geometry is hyperbolic. 
	We prove:
	
	\begin{theorem}\label{t:hyperbolic}
		Let $\Gamma$ be a lattice in $\isom(\H^n)$ for $n \geqslant 2$. 
		Let $\mu$ be a probability distribution on $\Gamma$ with finite first $\H^n$-moment such that the support of $\mu$ generates a subgroup that contains a pair of loxodromic elements with distinct axes.
		Then, with respect to the harmonic measure defined by the $\mu$-random walk on $\Gamma$, the space $\H^n$ with the hyperbolic metric is statistically hyperbolic. 
	\end{theorem} 
	
	We will present the proof of Theorem~\ref{t:hyperbolic} after the proof of the harder case of Theorem~\ref{t:main}. 
	This lets us use the geodesic separation property for the harmonic measure formulated in the earlier proof.
	We note that when $n > 2$ and $\Gamma$ is a non-uniform lattice, Randecker-Tiozzo proved that a harmonic measure arising from a $\mu$ whose support generates $\Gamma$ and has finite $(n-1)^{\text{th}}$ moment with respect to a word metric, is singular. 
	See~\cite[Theorem 2]{Ran-Tio}. 
	For uniform lattices, many classes of harmonic measures are known or conjectured to be singular. 
	For instance, a famous conjecture of Guivarc'h-Kaimanovich-Ledrappier asserts that harmonic measures that arise from finitely supported random walks on a uniform lattice in $SL(2,\R)$ are singular. 
	See \cite{Der-Kle-Nav}.
	So Theorem~\ref{t:hyperbolic} has new content.
	
	From now on we assume that the complex dimension of $\T(S)$ is greater than one and present Theorem~\ref{t:main} with that assumption. 
	
	\subsection{Strategy of the proof} 
	To derive statistical hyperbolicity, Dowdall-Duchin-Masur set up two properties to check. 
	The first property is called the thickness property. 
	See \cite[Definition 5.4]{Dow-Duc-Mas}. 
	It states that as the radius of a sphere goes to infinity a typical radial geodesic segment spends a definite proportion of its time in the thick part of \teichmuller space. 
	The second property is called the separation property. 
	See \cite[Definition 6.1]{Dow-Duc-Mas}. 
	It states that as the radius of a sphere goes to infinity a typical pair of radial geodesic segments exhibit good separation. 
	For Lebesgue-class visual measures, the ergodicity of the \teichmuller geodesic flow is the key tool in their proof of the thickness property. 
	For rotationally invariant Lebesgue measures, they verify the separation property by disintegrating the measure along and transverse to \teichmuller discs and then use the hyperbolic geometry of these discs. 
	
	For random walks, different tools are needed. 
	The main tool is the ergodicity of the shift map on the space of bi-infinite sample paths. 
	This ergodicity can be leveraged to prove that a typical bi-infinite sample path recurs to a neighbourhood of its tracked geodesic with a positive asymptotic frequency.
	As sample paths lie in a thick part, recurrence implies that the tracked geodesics spend a positive proportion of their time in a thick part.
	By tweaking the size of the neighbourhood, and hence the thick part, we show that the time spent in the thick part by the tracked geodesic can achieve any positive proportion. 
	While a positive proportion of thickness is suggested by the main theorem in \cite{Gad-Mah-Tio}, the precise quantitative version that we need here requires some work.
	
	For the separation property, we project two fellow travelling radial geodesic segments to the curve complex. 
	By a theorem of Maher-Tiozzo, a typical sample path makes linear progress in the curve complex.
	Combining this theorem with the recurrence, we show that the projections of fellow travelling radial geodesic segments must nest into a shadow. 
	 Also by a proposition in Maher-Tiozzo,  the harmonic measure of a shadow tends to zero in the distance of the shadow from the base-point. 
	This then enables us to conclude the required separation property.
	
	\section{Preliminaries} 
	
	\subsection{Statistical hyperbolicity:} 
	Let $(X,d)$ be a metric space.
	Let $x \in X$.
	Let $r > 0$.
	We call the set $S_r(x) = \{ x' \in X \text{ such that } d(x,x') =r\}$ the \emph{sphere} of radius $r$ centred at $x$.
	Suppose $\nu_r$ is a family of probability measures supported on $S_r(x)$. 
	Provided the limit exists, one defines a numerical index $E(X):=E(X,x,d,\{\nu_{r}\})$ by 
	\[
	E(X) = \lim_{r \to \infty} \frac{1}{r} \int\limits_{S_r(x) \times S_r(x) } d(x', x'') \, d \nu_r(x') d\nu_r(x'') 
	\]
	A space is said to be \emph{statistically hyperbolic} if $E(X) = 2$. This is motivated by the fact that $E(\H^{n})=2$ for any dimension equipped with the natural measures on spheres. Moreover, it was demonstrated by Duchin-Leli{\`e}vre-Mooney~\cite[Theorem 4]{Duc-Lel-Moo} that $E(G)=2$ for any non-elementary hyperbolic group $G$ with any choice of generating set. 
	
	For uniform lattices in $\isom(\H^n)$, the Green metric defined by the random walk is quasi-isometric to the induced hyperbolic metric through the orbit map. This suggests a derivation of statistical hyperbolicity by reducing the problem to the Duchin-Leli{\`e}vre-Mooney result. As our proof of Theorem \ref{t:hyperbolic} covers both uniform and non-uniform lattices, we omit the details for this alternate approach. In any case, it would work only for uniform lattices.
	
	We direct the reader to~\cite{Duc-Lel-Moo} for further discussion on the sensitivity of $E$. Indeed, it is not quasi-isometrically invariant, and has dependence on the base-point $x$ and the choice of measures $v_r$. Furthermore, $\delta$-hyperbolicity and exponential volume growth are not sufficient to guarantee statistical hyperbolicity.
	
	\subsection{Background on \teichmuller spaces} 
	Let $S$ be a surface of finite type, that is, $S$ is an oriented surface with finite genus and finitely many marked points.
	The \emph{\teichmuller space} $\T(S)$ of $S$ is the space of marked conformal structures on $S$. 
	By the Uniformisation Theorem, if $S$ has negative Euler characteristic then there is a unique hyperbolic metric in each marked conformal class.
	The \emph{mapping class group} $\Mod(S)$ is the group of orientation preserving diffeomorphisms of $S$ modulo isotopy.
	The mapping class group $\Mod(S)$ acts on $\T(S)$ by changing the marking. 
	The quotient space $\mathcal{M}(S) = \T(S) / \Mod(S)$ is the \emph{moduli space of Riemann surfaces}.
	Given $\epsilon > 0$, a marked hyperbolic surface $x \in \T(S)$ is $\epsilon$-\emph{thick} if the hyperbolic length of every closed geodesic on $x$ is at least $\epsilon$. 
	We let $\T_\epsilon(S)$ be the subset of $\T(S)$ consisting of all $\epsilon$-thick marked hyperbolic surfaces.
	We note that there exists an $\epsilon> 0$ that depends only on the complexity of the surface such that $\T_\epsilon(S)$ is non-empty. 
	See the discussion on the Bers constant in \cite[Chapter 12 Section 4.2]{Far-Mar}. 
	Observe that if $x$ is $\epsilon$-thick then so is $gx$ for any mapping class $g$. 
	Hence, we deduce that we get a thick-thin decomposition of the moduli space $\mathcal{M}(S)$. 
	The Mumford compactness theorem says that for any $\epsilon > 0$, the thick part $\mathcal{M}_\epsilon(S)$ is compact.
	
	Given a marked conformal surface $x$, let $\mathcal{Q}(x)$ be the set of meromorphic quadratic differentials on $x$ with simple poles at and only at the marked points. 
	This gives a bundle $\mathcal{Q}$ over $\T(S)$. 
	This bundle is stratified by the orders of the zeroes of the quadratic differentials.
	By contour integration and a choice of square root, each $q \in \mathcal{Q}(x)$ defines a half-translation structure on $S$. 
	That is, it defines charts to $\mathbb{C} = \mathbb{R}^2$ with half-translation transition functions of the form $z \to \pm z + c$. 
	It then makes senses to impose the condition that the half-translation surfaces that we consider have unit area.
	Given a stratum of quadratic differentials, one may fix a basis for the homology of $S$ relative to the marked points and zeroes. 
	One can associate a period to each basis element of the homology by integrating a square root of the quadratic differential over a contour representing it. 
	These periods give local co-ordinates on the stratum and can be used to define the Lebesgue measure class on it.
	The principal stratum is the stratum of quadratic differentials whose zeroes are all simple. 
	The Lebesgue measure class on the principal stratum can be pushed down to define a Lebesgue measure class on the \teichmuller space $\T(S)$.
	
	The $SL(2,\R)$-action on $\mathbb{R}^2$ preserves the area and also the form of the transition functions.
	Hence, it descends to an action on $\mathcal{Q}$.
	The compact part $SO(2,\R)$ acts by rotations and preserves the conformal structure.
	The diagonal part of the action given by 
	\[
	\left[ \begin{array}{cc} e^{t/2} & 0 \\ 0 & e^{-t/2} \end{array} \right]
	\]
	is called the \emph{\teichmuller flow} and we will denote it by $\phi_t$. 
	Given a pair $x, y$ of marked hyperbolic surfaces, Teichm{\"u}ller's theorem states that there is a unit area quadratic differential $q$ on $x$ and a time $t$ such that $\phi_t q$ projects to $y$. 
	The time $t$ is called the \emph{\teichmuller distance} between $x$ and $y$. 
	
	The \teichmuller distance gives a Finsler metric on $\T(S)$ which we will call the \emph{\teichmuller metric} and denote by $d_{\teich}$. 
	The mapping class group acts by isometries and hence the \teichmuller distance descends to $\mathcal{M}$. 
	Masur-Wolf \cite{Mas-Wol} showed that \teichmuller space with the \teichmuller metric is not a $\delta$-hyperbolic space. 
	This adds value to the question of statistical hyperbolicity.
	Dowdall-Duchin-Masur showed that for measures in the Lebesgue-class, $\T(S)$ is statistically hyperbolic. 
	See \cite[Theorems B, C and D]{Dow-Duc-Mas}.
	
	\subsection{Random walks on the mapping class group}
	Let $G$ be a finitely generated group. 
	Let $\mu$ be a probability measure on $G$.
	A sample path $w_n$ of length $n$ for the $\mu$-random walk on $G$ is a random product $w_n = g_1 g_2 \cdots g_n$ where each $g_i$ is independently sampled by $\mu$.
	The $n$-fold convolution $\mu^{(n)}$ of $\mu$ gives the distribution of $w_n$.   
	If $G$ has an action on a space $X$, one can use the orbit of a base-point to project the random walk onto $X$. 
	We are interested in the limiting behaviour of sample paths as $n \to \infty$. 
	For this reason, we will consider the shift on $G^{\N}$.
	It is convenient to consider both forward and backward random walks. 
	The backward random walk is simply the random walk with respect to the reflected measure $\hat{\mu}$ defined by $\hat{\mu}(g) = \mu(g^{-1})$. 
	We then consider bi-infinite sequences as elements of $G^{\Z}$ with the shift $\sigma$ acting as a step of the random walk.
	For the push-forward $h$ of the product measure $\mu^\Z$ on $G^\Z$ the conditional measure for the shift is given by $\mu$. 
	We can separate the forward and backward directions to write $h$ as the product $\nu \times \hat{\nu}$.
	We call the measure $\nu$ the harmonic measure. 
	
	By the Nielsen-Thurston classification, a mapping class is finite order, reducible or pseudo-Anosov. 
	A finite order mapping class is an automorphism of some Riemann surface.
	A reducible mapping class has some power that fixes a multi-curve on the surface. 
	A pseudo-Anosov mapping class $f$ has a \teichmuller axis: an $f$-invariant bi-infinite \teichmuller geodesic along which the map translates realising the infimum of $d_{\teich}(x, f(x))$ over $\T(S)$ by this translation.
	This description of a pseudo-Anosov map implies that the \teichmuller axis is unique and that its vertical and horizontal measured foliations are uniquely ergodic. 
	
	A subgroup of $\Mod(S)$ is non-elementary if it contains a pair of pseudo-Anosov mapping classes with distinct \teichmuller axes. Let $x \in \T(S)$ be a base-point. 
	Kaimanovich-Masur showed that if the support of a probability distribution $\mu$ on $\Mod(S)$ generates a non-elementary subgroup then for $h$-almost every sample path $\omega = (w_n)$ the sequence $w_n x$ converges to a projective class of a measured foliation on $S$. 
	See \cite[Theorem 2.2.4]{Kai-Mas1}. 
	Thurston showed that there is a natural way in which the space $\pmf(S)$ of projective classes of measured foliations serves as a boundary of $\T(S)$. 
	So the Kaimanovich-Masur theorem can be rephrased as convergence to the boundary $\pmf(S)$ for $h$-almost every sample path.
	In particular, the measure $\nu$ on $\Mod(S)^{\N}$ pushes forward to a measure on $\pmf(S)$. 
	We call this the harmonic measure on $\pmf(S)$.
	
	The distribution $\mu$ has \emph{finite first moment} with respect to the \teichmuller metric if 
	\[
	\sum\limits_{g \in \Mod(S)} \mu(g) d_{\teich} (x, g x) < \infty.
	\]
	By \cite[Theorem 1.2]{Mah-Tio}, finite first moment in the curve complex is sufficient for positive linear drift of typical sample paths when projected to the curve complex using the mapping class group action. 
	Since \teichmuller distance is a coarse upper bound for the curve complex distance, finite first $d_{\teich}$-moment implies finite first moment in the curve complex.
	Hence we have positive linear drift of sample paths in the curve complex and consequently in \teichmuller space. 
	
	Going back to the work of Kaimanovich-Masur, they showed that if $\mu$ has finite entropy and finite first logarithmic moment with respect to the \teichmuller metric then the push-forward measure is measurably isomorphic to $\nu$. 
	See \cite[Theorem 2.3.1]{Kai-Mas1}.
	For this reason, and to keep the notation simple, we denote the measure on $\pmf(S)$ by $\nu$ even though we do not need the measurable isomorphism.

	\subsection{Statistical hyperbolicity for a harmonic measure}
	
	Let $\mathcal{Q}^1(x)$ be the set of unit area quadratic differentials for the marked Riemann surface $x$. 
	When $q \in \mathcal{Q}^1(x)$ has a uniquely ergodic vertical foliation, the \teichmuller ray $\phi_t q$ converges as $t \to \infty$ to the projective class of the vertical foliation.
	Since $\nu$ is supported on the set of uniquely ergodic foliations, $\nu$ can be pulled back to a measure on 
	$\mathcal{Q}^1(x)$. 
	This gives us a measure on every sphere $S_r(x)$. 
	Thus, it makes sense to consider whether $\T(S)$ with the \teichmuller metric is statistically hyperbolic with respect to harmonic measure.
	
	\subsection{Statistical hyperbolicity in \teichmuller space}
	
	Dowdall-Duchin-Masur~\cite{Dow-Duc-Mas} reduce statistical hyperbolicity of \teichmuller space with the \teichmuller metric 
	for a family of measures $\{\nu_r\}$ to the verification of two properties: the thickness property~\cite[Definition 5.2]{Dow-Duc-Mas} and the separation property~\cite[Definition 6.1]{Dow-Duc-Mas}. 
	We will now state these properties and in Section~\ref{s:stat-hyp}, we will give a quick sketch of how these properties imply statistical hyperbolicity. For those unfamiliar with Dowdall-Duchin-Masur, we recommend reading through the sketch immediately after Definitions \ref{d:thick} and \ref{d:separate}.
	
	For a choice of $\epsilon>0$ and a geodesic segment $[x,x']\subset\T(S)$, we denote the proportion of time $[x,x']$ spends in $\T_\epsilon(S)$ by
	\[\thick_\epsilon^{\%}([x,x']):=\frac{|\{0 \leqslant t \leqslant d_{\teich}(x,x')\,:\,x'_s \in\T_\epsilon(S)\}|}{d_{\teich}(x,x')},\]
	where $x'_s$ is the point at distance $s$ from $x$ along $[x,x']$.
	
	The thickness property is the following.
	
	\begin{definition}[Thickness property]\label{d:thick} 
		A family of measures $\{\nu_r\}$ on spheres in $\T(S)$ has the thickness property if for all $0<\theta,\eta<1$ there exists an $\epsilon>0$ such that
		\[\lim_{r\to\infty}\frac{\nu_r\left(\{x'\in S_r(x)\,\,|\,\,\thick_\epsilon^{\%}([x,x'_t])\geqslant\theta\text{ for all }t\in[\eta r,r]\}\right)}{\nu_r(S_r(x))}=1,\]
		for all $x\in\T(S)$.
	\end{definition}
	
	The separation property is the following.
	
	\begin{definition}[Separation property]\label{d:separate} 
		A family of measures $\{\nu_r\}$ on spheres in $\T(S)$ has the separation property if for all $M>0$ and $0<\eta<1$, we have
		\[\lim_{r\to\infty}\frac{\nu_r\times\nu_r\left(\{(x',x'')\in S_r(x)\times S_r(x)\,\,|\,\,d_{\teich}(x'_t,x''_t)\geqslant M\text{ for all }t\in[\eta r,r]\}\right)}{\nu_r\times\nu_r(S_r(x)\times S_r(x))}=1,\]
		for all $x\in\T(S)$.
	\end{definition}
	
	In the next section, we derive these properties for a harmonic measure that arises from a random walk on the mapping class group determined by a probability distribution with finite first moment with respect to the \teichmuller metric whose support generates a non-elementary subgroup.
	
	\section{Derivation of the Thickness and Separation Properties}
	
	\subsection{Recurrence}
	Let $x$ be a base-point in \teichmuller space. 
	Let $\omega$ be a bi-infinite sample path. 
	As a convenient notation, we let $x_n = w_n x$ for any $n \in \Z$. 
	In particular, this means that $x_0$ is the same as the base-point $x$. 
	For almost every $\omega$, the sequences $x_n$ and $x_{-n}$ as $n \to \infty$ converge projectively to distinct uniquely ergodic measured foliations $\lambda^+$ and $\lambda^-$ respectively.
	For such sample paths, let $\gamma_\omega$ be the bi-infinite \teichmuller geodesic between $\lambda^+$ and $\lambda^-$. 
	As convenient notation, let $\gamma = \gamma_\omega$ and let $\gamma_n $ be a point of $\gamma$ that is closest to $x_n$. 
	The diameter of the set of closest points is coarsely bounded above by $d_{\teich}( x_n, \gamma_n)$. As sample paths deviate sub-linearly from their tracked geodesics, the choice of the closest point does not affect our estimates. 
	This will become quantitatively precise subsequently.

	Let
	\[
	\Lambda_R = \left\{ \omega \text{ such that } d_{\teich} (x, \gamma_\omega) < R \right\} 
	\]
	By \cite[Lemma 1.4.4]{Kai-Mas1}, the function $\omega \to d_{\teich} (x, \gamma_\omega)$ is measurable. 
	Recall that $h$ is our notation for the harmonic measure $\nu \times \hat{\nu}$ on bi-infinite sample paths.
	So if $R$ is large enough then $h(\Lambda_R) > 0$ and $h(\Lambda_R) \to 1$ as $R \to \infty$. 
	
	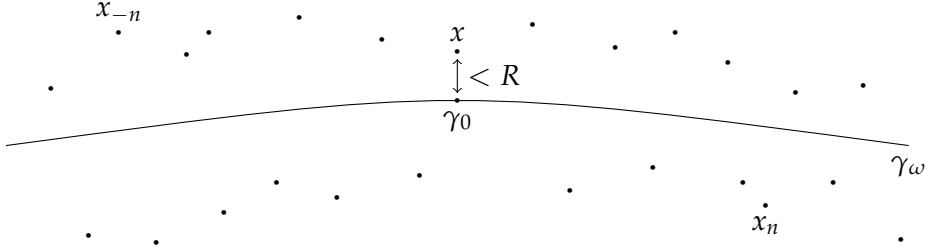
\begin{figure}[H]
		\begin{center}
			\begin{tikzpicture}
			\node[above] at (0,-0.25) {$x$};
			\node at (0,-0.25) {$\boldsymbol{\cdot}$};
			\draw (-6,-1.5) .. controls (0,-0.7) .. (6,-1.5);
			\node[below] at (6,-1.5){$\gamma_\omega$};
			\node at (0,-0.9){$\boldsymbol{\cdot}$};
			\node[below] at (0,-0.9){$\gamma_0$};
			\draw[<->] (0,-0.35)--node[right]{$<R$}(0,-0.8);
			\foreach \x/\y in {1/0.1,2.1/-0.2,2.9/0,3.6/-0.4,4.5/-0.8,5.4/-0.7,
				-1/-0.1,-2.1/0.2,-3.3/0,-3.6/-0.3,-4.5/0,-5.4/-0.75}
			{\node at (\x,\y){$\boldsymbol{\cdot}$};
				\node at (-\x+0.5,-2+\y){$\boldsymbol{\cdot}$};}
			\node[below] at (4.1,-2.3){$x_n$};
			\node[above] at (-4.5,0){$x_{-n}$};
			\end{tikzpicture}
		\end{center}
		\caption{A sample path $\omega$ in $\Lambda_{R}$.}
	\end{figure}\label{Lambda_R}
	
	An integer $k$ will be called an \emph{$R$-recurrence time} for $\omega $ if $\sigma^k \omega \in \Lambda_R$ where recall that $\sigma$ is the shift map.
	Suppose $j < k$ are $R$-recurrence times for $\omega$ and suppose $d_{\teich}(x_j, x_k) = 2d$. This distance will be bounded by the sum of the length of each step we do, that is,
	\[
	2d \leqslant \sum_{i=j+1}^{k} d_{\teich}(x_{i-1},x_i) = \sum_{i=j+1}^{k} d_{\teich}(x, g_i x). 
	\]
		
	We let $[\gamma_j, \gamma_k] $ be the segment of $\gamma_\omega$ connecting $\gamma_j $ and $\gamma_k$. 
	We note that
	\[
	\len [\gamma_j, \gamma_k] \leqslant 2R + 2d.
	\]
	If $2d \leqslant 2R$, then 
	\[
	[\gamma_j, \gamma_k] \subset B(x_j, 3R) \cup B(x_k, 3R).
	\]
	So suppose $2d > 2R$. 
	We consider the sub-segments of $[\gamma_j, \gamma_k]$ that might be outside of the union $B(x_j, 3R) \cup B(x_k, 3R)$. 
	We denote the union of these sub-segments by $C_{j,k}$ and let $L(j,k)$ to be the sum of their lengths.	
	
	Let $0 < \rho < p < 1$. 
	We choose $R$ large enough such that $h(\Lambda_R) \geqslant p$. 
	Let $n \in \N$ and set
	\[
	E^{(1)}_n = \left\{ \omega \text{ such that } \frac{1}{m} \sum\limits_{0 \leqslant k \leqslant m} \chi_{\Lambda_R}(\sigma^k \omega) < p - \rho \text{ for some } m \geqslant n \right\} 
	\]
	where $\chi$ is the indicator function.
	By ergodicity of the shift map $\sigma$ it follows that $h(E^{(1)}_n )  \to 0$ as $n \to \infty$.
	
	Suppose $\omega$ is in the complement of $E^{(1)}_n$.
	Then the number of times $i \in \{ 0, \cdots, n\}$ such that $\sigma^i \omega \notin \Lambda_R$ is at most $(1+\rho-p) n$.
	Let $j_{\min} $ and $j_{\max}$ be the smallest and largest $R$-recurrence times in $\{0,\dots, n\}$. 
	Then we note that
	\[
	j_{\min} \leqslant (1+\rho-p) n \hskip 8pt \text{and} \hskip 8pt j_{\max} \geqslant n - (1+\rho - p) n.
	\]
	By bounding with steps, we get the estimate
	\[
	d_{\teich}(x ,x_{j_{\min}}) \leqslant \sum_{i=1}^{j_{\min}}d_{\teich}(x,g_ix).
	\]
	
	We will separate the sum into two sums. The first will contain terms $d_{\teich} (x, g_i x)$ for which $d_{\teich} (x, g_i x) \leqslant D$ and the second will contain the rest of the terms. 
	For convenience of notation, we let $B$ be the set of bi-infinite sample paths $\omega$ whose first step $g_1$ satisfies $d_{\teich}(x, g_1 x) \leqslant D$. 
	The set $B$ depends on the choice of $D$ but we will suppress this from the notation for the moment and point it out when required later. 
	With this notation, the sum above becomes 
	\[
	d_{\teich}(x ,x_{j_{\min}}) \leqslant \sum_{i=1}^{j_{\min}}d_{\teich}(x,g_ix)\chi_B (\sigma^i(\omega))  + \sum_{i=1}^{j_{\min}}d_{\teich}(x,g_ix)\chi_{\Omega \setminus B} (\sigma^i (\omega)),
	\]
	
	We can bound the first sum by $\sum_{i=1}^{j_{\min}} D \leqslant (1+\rho - p)D n$, and the second one by
	\[
	 \sum_{i=1}^{n} d_{\teich}(x,g_ix)\chi_{\Omega \setminus B} (\sigma^i(\omega)).
	\]
	
	We denote each of the terms of the sum above as $b_i^D$, that is $b_i^D(\omega) =  d_{\teich}(x,g_ix)\chi_{\Omega \setminus B} (\sigma^i(\omega))$. 
	The random variables $b_i^D$ are all independent and identically distributed. 
	Furthermore, $b_i^D( \omega) \leqslant d_{\teich}(x,g_ix)$, which, since the measure has finite first $d_{\teich}$-moment, is integrable. 
	By the strong law of large numbers, the sum above when divided by $n$ converges almost surely to $\mathbb{E}[b_i^D]$. 
	Let $C$ denote this expectation. 
	
	We conclude that the sets
	\[
	E^{(2)}_n = \left\{ \omega \text{ such that } \frac{1}{m} \sum_{i=0}^{m}b_i^D(\omega) > C+c \text{ for some } m \geqslant n \right\} 
	\]
	satisfy $h(E^{(2)}_n)\to 0$ as $n\to \infty$ for all $c>0$. 
	
	As $D$ tends to infinity, the random variable $b_i^D$ converges to zero point-wise. By the dominated convergence theorem where we are dominating by $d_{\teich}(x,gx)$, we get that if $D$ is large enough, then the expectation $C$ is small.
	
	Assume $\omega$ is in the complement of both $E^{(1)}_n$ and $E^{(2)}_n$. Then,
	\[
	d_{\teich}(x ,x_{j_{\min}}) \leqslant ((1+\rho - p)D +C+c)n ,
	\]
	and, by the same reasoning,
	\[
	d_{\teich}(x_{j_{\max}} ,x_{n}) \leqslant ((1+\rho - p)D +C+c)n.
	\]
	This implies
	\[
	d_{\teich} (\gamma_0, \gamma_{j_{\min}}) \leqslant ((1+\rho - p)D +C+c)n + 2R
	\]
	and 
	\[
	d_{\teich} (\gamma_{j_{\max}}, \gamma_n) \leqslant ((1+\rho - p)D +C+c)n + 2R.
	\]

        We will now estimate from above the time that the segment $[ \gamma_{j_{\min}}, \gamma_{j_{\max}}]$ spends outside the neighbourhood of points along the sample path. 
	A pair $j < k$ of recurrence times is \emph{consecutive} if every $J$ satisfying $j < J < k$ is not a recurrence time. 
	
	If $k=j+1$, the set $C_{j, j+1}$ will be non empty only if $d_{\teich}(x_j,x_{j+1})\geqslant 2R$.
	 In that case, if we choose $D$ smaller than $R$, then $L(j,j+1)\leqslant b_{j+1}^D$. 
	 
	 If $k>j+1$, we have $k-j-1$ steps taken outside the $R$-neighbourhood of the geodesic.
	 We can split these steps depending on whether they are in $B$ or not. 
         Recall that since we are outside the exceptional set $E^{(1)}_n$, we know that the total number of non-recurrence times is bounded above by $(1+\rho-p)n$. 
	 In the estimate for the sum of $L(j,k)$, we note that each non-recurrence time contributes to at most two terms in the sum. 
	We deduce that
	\begin{align*}
	\sum\limits_{\substack{j < k \\\text{consecutive}}} L(j,k)
	&\leqslant
	\sum\limits_{\substack{ j+1 = k \\\text{consecutive}}} b_{j+1}^D + \sum\limits_{\substack{j+1 < k \\\text{consecutive}}} \sum\limits_{i=j+1}^k  \left(b_i^D+d_{\teich}(x,g_ix)\chi_{B}(\sigma^i(\omega))\right) \\
	&\leqslant
	\sum_{i=1}^n b_i^D + 2(1+\rho-p)n D \\
	&\leqslant (2(1+\rho - p)D +C+c)n.
	\end{align*}

	\subsection{Linear progress} 
	The \teichmuller metric is sub-additive along sample paths. 
	By Kingman's sub-additive ergodic theorem, there exists a constant $A \geqslant 0$ such that for almost every sample path $\omega$ we have
	\[
	\lim_{n \to \infty} \frac{d_{\teich} (x_0 , x_n )}{n} = A.
	\]
	By \cite[Theorem 1.2]{Mah-Tio}, $A > 0$. 
	
	Let $0 < a < 1$ be a constant smaller than $A$. 
	Let $n \in \N$. 
	Consider the set of sample paths
	\[
	\Omega^{(3)}_n = \left\{ \omega \text{ such that } (A- a) m < d_{\teich}(x_0, x_m) < (A+a) m  \text{ for all } m \geqslant n  \right\}.
	\]
	Let $E^{(3)}_n$ be the complement $\Omega \setminus \Omega^{(3)}_n$. 
	It follows that $h(E^{(3)}_n) \to 0$ as $n \to \infty$. 
	
	\subsection{Thickness along tracked geodesics} 
	Let $0 < \theta' < 1$. 
	
	We parameterise the tracked geodesic $\gamma = \gamma_\omega$ by unit speed such that at time zero we are at $\gamma_0$, a closest point to $x_0 = x$,  and $\gamma(t) \to \lambda^+$ as $t \to \infty$. 
	
	Let $\Lambda(r, \theta', \epsilon')$ be the set of sample paths $\omega$ such that for all $s  > r$ we have 
	\[
	\thick^{\%}_{\epsilon'} [\gamma(0), \gamma(s)] \geqslant \theta'.
	\]
	
	\begin{proposition}\label{p:tracked-thick}
		Given $0 < \theta' < 1$ there exists an  $\epsilon' > 0$ such that 
		\[
		\lim_{r \to \infty} h(\Lambda(r, \theta', \epsilon')) =1.
		\]
	\end{proposition} 
	
	\begin{proof}
		Given $R > 0$ there exists $\epsilon(R) > 0$ such that $B(x_0, 3R) \subset \T_{\epsilon(R)}(S)$. 
		By equivariance, $B(x_n, 3R) \subset  \T_{\epsilon(R)}(S)$ for all $n \in \Z$. 
		
		Suppose that $\omega$ is in the complement of $E^{(1)}_n\cup E^{(2)}_n$. 
		We first prove the proposition along the discrete set of times $\gamma_n$ along $\gamma_\omega$. 
		By the triangle inequality
		\[
		d_{\teich}(\gamma_0, \gamma_n) \geqslant d_{\teich}(x_0, x_n) - d_{\teich}(x_0, \gamma_0) - d_{\teich}(x_n, \gamma_n).
		\]
		Since $\gamma_0$ is the closest point in $\gamma_\omega$ to $x_0$
		\[
		d_{\teich}(x_0, \gamma_0) \leqslant d_{\teich} (x_0, x_{j_{\min}} ) + R \leqslant ((1+\rho - p)D +C+c)n  + R .
		\]
		Similarly
		\[
		d_{\teich}(x_n, \gamma_n) \leqslant d_{\teich} (x_n, x_{j_{\max}} ) + R \leqslant  ((1+\rho - p)D +C+c)n + R .
		\]
		So we get
		\[
		d_{\teich}(\gamma_0, \gamma_n) \geqslant d_{\teich}(x_0, x_n) - 2((1+\rho - p)D +C+c)n  -2R.
		\]
		
		Further assume that $\omega$ is also in the complement of $E^{(3)}_n$.
		We deduce from the above estimates that
		\begin{equation}\label{e:lower-bound}
		d_{\teich}(\gamma_0, \gamma_n) \geqslant (A-a)n -  2((1+\rho - p)D +C+c)n  -2R.
		\end{equation}
		
		We note that the points in the segment $[\gamma_0, \gamma_n]$ that are not in $\T_{\epsilon(R)}(S)$ are in the union of the sets $[\gamma_0, \gamma_{j_{\min}}]$, $[\gamma_{j_{\max}}, \gamma_n]$ and the sets $C_{j,k}$ for all consecutive recurrence pairs $j < k$. 
		The individual upper bounds on the lengths of each set in the union gives us  the bound on the thick proportion for a choice of $\epsilon' \leqslant \epsilon(R)$
		\begin{equation}\label{e:thick-est} 
		1- \thick^{\%}_{\epsilon'} [\gamma_0, \gamma_n]  \leqslant \frac{ 4((1+\rho - p)D +C+c)}{(A-a) - 2((1+\rho - p)D +C+c)  - 2R/n}.
		\end{equation} 
		Now we make explicit choices as follows. 
		\begin{itemize} 
		\item Note that as $R \to \infty$, the proportion $p \to 1$ and hence $\min \{ (1-p)^{-1/2}, R \} \to \infty$. 
		\item Recall that if we tend $D \to \infty$ then $C \to 0$. 
		\item So if we set $D = \min \{ (1-p)^{-1/2}, R \}$ then we can pass to $R$ large enough so that $C$ is small enough. 
		\item Furthermore, note that $D(1-p) \leqslant (1-p)^{1/2}$ which can also be made small with $R$ large enough.
		\item Recall also that $\rho, c$ and $a$ can be chosen at the outset and set to be small enough.
		\end{itemize} 
		These choices imply that by choosing $R$ large enough we may arrange that the numerator and the second term in the denominator in (\ref{e:thick-est}) are small. 
		Once $R$ has been fixed the third term in the denominator goes to zero as $n \to \infty$. 
		In conclusion, by choosing $R$ large enough we may arrange the right hand side of (\ref{e:thick-est}) to be smaller than $1-\theta'$.
		
		Now we conclude the proposition as the time $s \to \infty$ along $\gamma_\omega$. 
		By an argument identical to the derivation of~(\ref{e:lower-bound}), we get the upper bound
		\[
		d_{\teich}(\gamma_0, \gamma_n) \leqslant (A+a)n +2((1+\rho - p)D +C+c)n + 2R.
		\]
		Given a time $s> 0$, we may choose $n$ to satisfy 
		\[
		(A-a)n - 2((1+\rho - p)D +C+c)n  -2R < s < (A+a)n + 2((1+\rho - p)D +C+c)n + 2R.
		\]
		When $s$ is large enough, such a choice always exists. 
		Since we are only interested in the limit as $s \to \infty$, we may make this choice. 
		Further tweaking $R$ and hence $p$ and also tweaking $\rho, c$ and $a$ we may arrange that the ratio of the upper bound to the lower bound in the above inequality is as close to one as we want. 
		This implies that as $s \to \infty$ the thick proportion of $[\gamma_0, \gamma(s)]$ is the same as the thick proportion of $[\gamma_0, \gamma_n]$.
		
		Finally, we note that the set of exceptions is the union $E^{(1)}_n \cup E^{(2)}_n\cup E^{(3)}_n$ whose measure tends to zero as $n \to \infty$. 
		In particular, this implies $h(\Lambda(r, \theta', \epsilon')) \to 1$ as $r \to \infty$, and we are done.
	\end{proof}
	
	As a direct consequence of Proposition~\ref{p:tracked-thick}, we get the following conclusion.
	
	\begin{proposition}\label{p:tracked-thick-2}
		Let $0 < \theta' < 1$. Then there exists an $\epsilon' > 0$ such that for almost every bi-infinite sample path $\omega$ there exists $t_\omega$ such that for all $t > t_\omega$ 
		\[
		\thick^{\%}_{\epsilon'} [\gamma(0), \gamma(t)] \geqslant \theta'.
		\]
	\end{proposition}
	
	\subsection{Thickness along rays} 
	Now let $y$ be some other base-point in $\T(S)$ possibly distinct from the base-point $x_0$ for the random walk.
	Masur proved that \teichmuller rays with the same vertical foliation are asymptotic if the foliation is uniquely ergodic. 
	See~\cite[Theorem 2]{Mas}. 
	We now use this result to transfer the thickness estimates from tracked geodesics to corresponding rays from $y$. 
	Suppose that $\omega$ is a typical bi-infinite sample path with the tracked geodesic $\gamma_\omega$. 
	Let $\lambda^+_\omega$ be the projective measured foliation that $\gamma_\omega$ converges to in the forward direction.
	Let $\xi_\omega$ be the geodesic ray from $y$ that converges to $\lambda^+_\omega$. 
	We may parameterise $\xi_\omega$ with unit speed so that $\xi_\omega(0) =y$ and $\xi_\omega(t) \to \lambda^+_\omega$ as $t \to \infty$. 
	
	\begin{proposition}\label{p:thick-ray}
		Let $0 < \theta < 1$. Then there exists $\epsilon > 0$ such that for almost every bi-infinite sample path $\omega$ there is a time $T_\omega > 0$ such that 
		\[
		\thick^{\%}_{\epsilon} [\xi_\omega (0), \xi_\omega (t)] \geqslant \theta
		\]
		for all $t  \geqslant T_\omega$. 
	\end{proposition} 
	
	\begin{proof}
		By \cite[Theorem 2.2.4]{Kai-Mas1}, for almost every $\omega$ the foliation $\lambda^+_\omega$ is uniquely ergodic.  
		By Masur's theorem, there is a time $s > 0$ that depends only on $d_{\teich}(\gamma_\omega, y)$ such that for all $t \geqslant s$ we have $d_{\teich}(\xi_\omega(t), \gamma_\omega) < 1/2$. 
		We may choose $\epsilon \leqslant \epsilon'$ such that the 1-neighbourhood of $\T_{\epsilon'}(S)$ is contained in $\T_\epsilon(S)$. 
		This means after the time $s$ along $\xi_\omega$ any $\epsilon'$-thick segment of $\gamma_\omega$ gives an $\epsilon$-thick segment of $\xi_\omega$ of at least the same length.
		We now set $\theta' > \theta$ and use Proposition~\ref{p:tracked-thick-2}. 
		Let $t > s$. 
		The total length of $\epsilon'$-thick segments of $\gamma_\omega$ that are inside a 1-neighbourhood of $[\xi_\omega(s), \xi_\omega(t)]$ is at least $\theta' t - s - d_{\teich}(\gamma(0), y)$.
		If $t$ is large enough then $\theta' t - s - d_{\teich}(\gamma(0), y) > \theta t$, and we are done. 
	\end{proof} 
	
	As an immediate corollary, we get
	\begin{corollary}
		The measures $\nu_r$ on spheres arising from the harmonic measure satisfy the thickness property~\ref{d:thick}.
	\end{corollary} 
	
	\subsection{Separation properties}\label{s:separate}

	To prove the separation properties, we project to the complex of curves. 
	Under this projection, \teichmuller geodesics give unparameterised quasi-geodesics and thick segments make coarse linear progress. 
	If a pair of \teichmuller rays with good thickness properties up to a distance $r$ are not well-separated then their projections to the curve complex fellow travel.
	Hence, the endpoints of their projections lie in some shadow and this shadow is pushed further and further out as $r \to \infty$. 
	So roughly speaking, we may conclude the separation properties from knowing that the harmonic measure of these shadows goes to zero.
	We now give the details of the argument. 

	The complex of curves $\cC(S)$ is a graph whose vertices are isotopy classes of essential simple closed curves on $S$. 
	Two vertices $\alpha$ and $\beta$ have an edge between them if $\alpha$ and $\beta$ have representatives that are disjoint. 
	The curve complex $\cC(S)$ is a locally infinite graph with infinite diameter. 
	Masur-Minsky showed that $\cC(S)$ is in fact $\delta$-hyperbolic. See \cite[Theorem 1.1]{Mas-Min}.
	Recall that the Gromov product on $\cC(S)$ is defined as follows: given a base-point $x$ and two points $y, z$ in the space, the Gromov product of $y$ and $z$ based at $x$ is defined as 
	\[
	(y,z)_x = \frac{1}{2} \left( d_{\cC} (x,y) + d_{\cC}(x,z) - d_{\cC}(y,z) \right)
	\]
	where $d_{\cC}$ denotes the distance in the curve complex.
	Let $\partial \cC(S)$ be the Gromov boundary of $\cC(S)$. 
	Given a number $\tau > 0$, the shadow of $y$ from $x$ with distance parameter $d_{\cC}(x,y) - \tau$ is defined as 
	\[
	\shad_x (y, \tau) = \{ z \in \cC(S) \cup \partial \cC(S) \text{ such that } (y,z)_x \geqslant d_{\cC}(x,y) - \tau \}.
	\]
	There are different definitions of shadows in the literature and the one we use here is from Maher-Tiozzo. See \cite[Page 197]{Mah-Tio}.
	
	To every marked hyperbolic surface $x\in\T(S)$, one can consider a systole on $x$, that is, a shortest closed hyperbolic geodesic on $x$. 
	A systole is always a simple closed curve and hence can be thought of as a vertex in the curve complex $\cC(S)$.
	This defines a coarse projection $\sys: \T(S) \to \cC(S)$ from \teichmuller space to the curve complex.
	By \cite[Theorems 2.3 and 2.6]{Mas-Min}, the projection $\sys(\gamma)$ of a \teichmuller geodesic $\gamma$ is an unparameterised quasi-geodesic in $\cC(S)$ with uniform constants that depend only on the surface.
	
	Let $M > 0$. 
	Suppose $\xi_\omega$ and $\xi_\eta$ are geodesic rays from $y$ chosen with respect to the harmonic measure, where $\omega$ and $\eta$ are the associated bi-infinite sample paths. 
	Let $T_\omega$ and $T_\eta$ be the thresholds given by Proposition ~\ref{p:thick-ray} for $\omega$ and $\eta$, respectively. 
	Pick $T$ larger than $T_\omega$ and $T_\eta$. 
	Suppose that $d_{\teich}(\xi_\omega (T), \xi_\eta (T)) < M$.
	Since the projection by systoles is coarsely Lipschitz, we may continue to assume $d_{\cC}(\sys(\xi_\omega(T)), \sys(\xi_\eta(T))) < M$. 
	
	We consider the projections $\sys([y, \xi_\omega(T)])$ and $\sys([y, \xi_\eta(T)])$. 
	Since $\xi_\omega$ and $\xi_\eta$ spend at least $\theta$ proportion of their time in the thick part their projections to $\cC(S)$ make linear progress, that is there exists a constant $\kappa > 0$ such that
	$d_\cC(\sys(y), \sys(\xi_\omega(T))) \geqslant \kappa T$ and $d_{\cC} (\sys(y), \sys(\xi_\eta(T))) \geqslant \kappa T.$
	
	Denote by $\lambda^+_\eta$ the limiting point of $\sys(\xi_\eta)$. By hyperbolicity of $\cC(S)$ there is a constant $r > 0$ such that the Gromov product between $\lambda^+_\eta$ and $\sys(\xi_\omega(T))$ satisfies
	\[
	(\sys(\xi_\omega(T)),\lambda^+_\eta)_{\sys(y)}\geqslant d_{\cC}(\sys(y),\sys(\xi_\omega(T)))- M - r. 
	\]
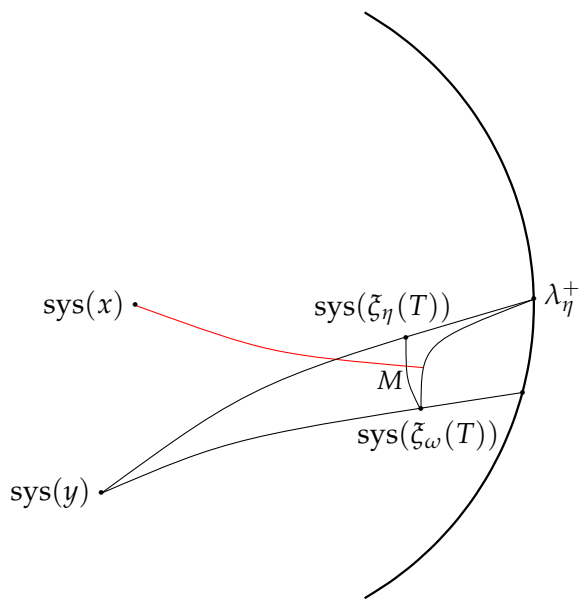
\begin{figure}[H]
	\begin{center}

		\begin{tikzpicture}
		\node at (-1.25,-2.5){$\boldsymbol{\cdot}$};
		\node[left] at (-1.25,-2.5){$\sys(y)$};
		\node at (3,-1.38){$\boldsymbol{\cdot}$};
		\node[below] at (3.1,-1.4){$\sys(\xi_\omega(T))$};
		\node at (2.8,-0.43){$\boldsymbol{\cdot}$};
		\node[above] at (2.5,-0.4){$\sys(\xi_\eta(T))$};
		\draw [red] (-0.8,-0) .. controls(1,-0.65) .. (3.04,-0.83);
		\node at (-0.8,0){$\boldsymbol{\cdot}$};
		\node[left] at (-0.8,0){$\sys(x)$};
		\draw (-1.25,-2.5) .. controls(0.5,-1.75) .. (4.35,-1.16);
		\draw (-1.25,-2.5) .. controls(0.8,-1) .. (4.5,0.08);
		\draw (3,-1.4) .. controls(2.8,-1) .. (2.8,-0.45);
		\node[left] at (2.9,-1){\small${M}$};
		\draw  (3,-1.4) .. controls(3,-0.5) .. (4.5,0.08);
		\draw [line width = 0.3mm,domain=-60:60] plot [smooth] ({4.5*cos(\x)}, {4.5*sin(\x)});
		\node at (4.35,-1.16){$\boldsymbol{\cdot}$};
		\node at (4.5,0.08){$\boldsymbol{\cdot}$};
		\node[right] at (4.5,0.08){$\lambda^+_\eta$};
		\end{tikzpicture}
	\end{center}
	\caption{The Gromov product $(\sys(\xi_\omega(T)),\lambda^+_\eta)_{\sys(x)}$ in $\cC(S)$, up to an additive constant.}
	\label{f:shadows}
\end{figure}

	In order to estimate harmonic measures, we will now pass to $\sys(x)$ as the base-point for Gromov products. 
	By the triangle inequality 
	\[
	d_{\cC}(\sys(x), \sys(\xi_\omega(T)) ) \geqslant d_\cC(\sys(y), \sys(\xi_\omega(T))) - d_{\cC}(\sys(x), \sys(y)).
	\]
	Note that $d_{\cC}(\sys(x), \sys(\xi_\omega(T)))$ goes to infinity as $d_\cC(\sys(y), \sys(\xi_\omega(T)))$ does. We have $(\sys(\xi_\omega(T),\lambda^+_\eta)_{\sys(x)}\geqslant d_{\cC} (\sys(x),\sys( \xi_\omega(T)))-r'$, where $r' = M + r  + 2 d_{\cC} (\sys(x), \sys(y))$. 
	Then $\lambda^+_\eta$ is contained in $\shad_{\sys(x)}(\sys( \xi_\omega(T)), r')$.
	By \cite[Proposition 5.1]{Mah-Tio}, the supremum of the harmonic measure of $\shad_{\sys(x)}(\sys( \xi_\omega(T)),r')$ tends to zero as $T \to \infty$. 
	
	As an immediate corollary we get
	\begin{corollary}\label{c:separate}
		The measures $\nu_r$ on spheres arising from the harmonic measure satisfy the separation property~\ref{d:separate}.
	\end{corollary} 
		
	For exceptional moduli or more generally for lattices in $\isom(\H^n)$, the ambient geometry is hyperbolic.
	So we may directly use shadows/ half-spaces in $\H^n$. 
	As above, the harmonic measure of shadows decays to zero as the distance from the base-point goes to infinity. 
	So the proof can be carried out exactly as above to conclude the separation property.

	\section{Statistical Hyperbolicity}\label{s:stat-hyp}  
	For the sake of completeness, we now sketch the spherical version of the argument given by Dowdall-Duchin-Masur~\cite[Theorem 7.1]{Dow-Duc-Mas} of how the thickness and separation properties imply statistical hyperbolicity. The idea is to mimic a proof of the fact that $E(\H^n)=2$ for the natural measures on spheres which makes use of the $\delta$-hyperbolicity of $\H^n$, and of the speed of separation of geodesics. 
	
	However, as discussed above, \teichmuller space is neither $\delta$-hyperbolic, nor negatively-curved in the sense of Busemann. The motivation for the separation property~\ref{d:separate} is to show instead that most pairs of geodesics after some threshold time become separated by a definite amount. This replaces the use of the negative curvature of $\H^n$. The combination of the following theorem of Dowdall-Duchin-Masur~\cite[Theorem A]{Dow-Duc-Mas} and the thickness property~\ref{d:thick} then replaces the use of the $\delta$-hyperbolicity of $\H^n$.
	
	\begin{theorem}[~\cite{Dow-Duc-Mas}, Theorem A]\label{t:theorem-A}
		For any $\epsilon>0$ and any $0<\theta'\leqslant1$, there exist constants $C$ and $L$ such that for any geodesic sub-interval $I\subset[x,x']\subset\T(S)$ of length at least $L$ and spending at least $\theta$ proportion of its time in $\T_\epsilon(S)$, we have
		\[I\cap\Nbhd_C([x,x'']\cup[x',x''])\neq\varnothing,\]
		for all $x''\in\T(S)$.
	\end{theorem}
	
	We now sketch the proof of statistical hyperbolicity using Theorem~\ref{t:theorem-A} with the thickness and separation properties, namely Definition~\ref{d:thick} and Definition~\ref{d:separate}. 
	
	Let $x\in\T(S)$. We choose $\theta$ large enough, say $\theta=3/4$, and then for any $0<p<\eta<1/3$ we let $\epsilon=\epsilon(\theta,\eta)>0$ be that guaranteed by the thickness property~\ref{d:thick}. Now choose $\theta'<\theta$, say $\theta'=1/2$, and let $C$ and $L$ be the constants given by Theorem~\ref{t:theorem-A} for our choice of $\theta'$ and $\epsilon$. The thickness and separation properties then imply that for all $r$ large enough, there exists a subset $P_r\subset S_{r}(x)\times S_{r}(x)$ whose complement has $\nu_r\times \nu_r$-measure at most $p$ and is such that for all $(x',x'')\in P_r$ we have $d_{\teich}(x'_t,x''_t)\geqslant 3C$, and
	\[\thick_\epsilon^{\%}([x,x'_t]),\thick_\epsilon^{\%}([x,x''_t])\geqslant\theta,\]
	for all $\eta r\leqslant t \leqslant r$. It can be checked that $B(x'_t,C)\cap[x,x'']=\varnothing$ for all such points $x'_t$.
	
	Choosing $r$ large enough, we can arrange that the interval $I_r=[x'_{\eta r},x'_{2\eta r}]$ spends at least $\theta'$ proportion of its time in $\T_\epsilon(S)$, has length at least $L$ and, since $\eta<1/3$, is contained in $[x,x']$. By applying Theorem~\ref{t:theorem-A}, we must have that $I_r\cap\Nbhd_C([x,x'']\cup[x',x''])\neq\varnothing$ and, since we have already noted that $I_r\cap\Nbhd_C([x,x''])=\varnothing$, it then follows that there exists a point in $[x',x'']$ at distance at most $2\eta r+C$ from $x$. Hence we have that
	\[d_{\teich}(x',x'')\geqslant (2-4\eta)r-2C,\]
	for all $(x',x'')\in P_r$. From which it follows that
	\begin{align*}
	E(X)&:=\lim_{r\to\infty}\frac{1}{r}\int_{S_r(x)\times S_r(x)}d_{\teich}(x',x'')\,\,d\nu_r(x')\,d\nu_r(x'') \\
	&\geqslant\liminf_{r\to\infty}\frac{1}{r}\int_{S_r(x)\times S_r(x)}d_{\teich}(x',x'')\,\,d\nu_r(x')\,d\nu_r(x'') \\
	&\geqslant\liminf_{r\to\infty}\frac{1}{r}(1-p)((2-4\eta)r-2C) \\
	&=(1-p)(2-4\eta).
	\end{align*}
	Hence, the result follows as $p$ and $\eta$ can be taken to be arbitrarily small.
	
	\subsection{Proof of Theorem~\ref{t:hyperbolic}}
	We now give a quick proof of Theorem~\ref{t:hyperbolic}. 
	The ambient geometry in $\H^n$ is already hyperbolic.
	So we can simply bypass the thickness discussion and directly invoke the separation property for the harmonic measure discussed at the end of Section \ref{s:separate}. 
	The proof of statistical hyperbolicity then follows the one in Section \ref{s:stat-hyp} and is simpler.
	
	\section*{Funding}
	
	This work was supported by the Engineering and Physical Sciences Research Council [EPSRC DTG EP/N509668/1 M\&S to L.J.] 
	
	\subsection{Acknowledgements:} 
	We thank Maxime Fortier-Bourque for conversations that clarified many points in this work. We also thank the referees for detailed comments that helped the paper.


\begin{thebibliography}{00}
		
		\bibitem{Bla-Hai-Mat} Blach\`{e}re, S., Ha\"{i}ssinsky, P. and Mathieu, P. {\em Harmonic measures versus quasi-conformal measures for hyperbolic groups.} Ann. Sci. \'{E}c. Norm. Sup\'{e}r (4) 44 (2011), no. 4, 683-721.
		
		\bibitem{Der-Kle-Nav} Deroin, B., Kleptsyn, V. and Navas, A. {\em On the question of ergodicity of minimal group actions on the circle.} Mosc. Math. J. 9 (2009), no. 2, 263-303.
		
		\bibitem{Dow-Duc-Mas} Dowdall, S., Duchin, M. and Masur, H. {\em Statistical hyperbolicity in \teichmuller space.} Geom. Funct. Anal. 24 (2014), no. 3, 748-795.
		
		\bibitem{Duc-Lel-Moo} Duchin, M., Leli{\`e}vre, S. and Mooney, C. {\em Statistical hyperbolicity in groups}, Algebr. Geom. Topol. 12 (2012), 1--18.
		
		\bibitem{Esk-Mir-Raf} Eskin, A., Mirzakhani, M. and Rafi, K. {\em Absolutely continuous stationary measures for the mapping class group.} Bistro seminar by Rafi, 23 April 2020.
		
		\bibitem{Far-Mar} Farb, B. and Margalit, D. {\em A primer on mapping class groups.} Princeton Mathematical Series, 49, Princeton University Press, Princeton NJ (2012).  
		
		\bibitem{Gad} Gadre, V. {\em Harmonic measures for distributions with finite support on the mapping class group are singular.} Duke Math. J. 163 (2014), no. 2, 309-368.
		
		\bibitem{Gad-Mah-Tio2} Gadre, V., Maher, J. and Tiozzo, G. {\em Word length statistics and Lyapunov exponents for Fuchsian groups with cusps.} New York J. Math. 21 (2015), 511-531.
		
		\bibitem{Gad-Mah-Tio} Gadre, V., Maher, J. and Tiozzo, G. {\em Word length statistics for \teichmuller geodesics and singularity of harmonic measure.} Comment. Math. Helv. 92 (2017), no. 1, 1-36.
		
		\bibitem{Gui-LeJ} Guivarc'h, Y. and LeJan, Y. {\em Sur l'enroulement du flot g\'{e}od\'{e}sique.} C. R. Acad. Sci. Paris. S\'{e}r. I Math. 311 (1990), 645-648.
		
		\bibitem{Kai-Mas1} Kaimanovich, V.A. and Masur, H. {\em The Poisson boundary of the mapping class group.} Invent. Math. 125 (1996), 221-264.
		
		\bibitem{Mah-Tio} Maher, J. and Tiozzo, G. {\em Random walks on weakly hyperbolic groups.} J. Reine Angew. Math. 742 (2018), 187-239.
		
		\bibitem{Hid} Masai, H. {\em On continuity of drifts of the mapping class group.} preprint (2019), arxiv.org/abs/1812.06651.
		
		\bibitem{Mas} Masur, H. {\em Uniquely ergodic quadratic differentials.} Comment. Math. Helv. 55 (1980), no. 2, 255-266. 
		
		\bibitem{Mas-Wol} Masur, H. and Wolf, M. {\em \teichmuller space is not Gromov hyperbolic.} Ann. Acad. Sci. Fenn. Ser. A I Math. 20 (1995), no. 2, 259-267.
		
		\bibitem{Mas-Min} Masur, H. and Minsky, Y. {\em Geometry of the complex of curves I: Hyperbolicity.} Invent. Math 138 (1999), no. 1, 103-149.
		
		\bibitem{Ran-Tio} Randecker, A. and Tiozzo, G. {\em Cusp excursion in hyperbolic manifolds and singularity of harmonic measure.} preprint (2019) arxiv.org/abs/1904.11581

		
	\end{thebibliography}
\end{document}